\documentclass[a4paper,11pt]{amsart}
\usepackage{verbatim}

\usepackage{amsfonts,amssymb}
\usepackage{graphicx,tikz}
\usepackage{tikz-cd}
\usepackage{float}
\usepackage{hyperref}
\usepackage[capitalise]{cleveref}
\usepackage[toc,page]{appendix}
\usepackage{amsmath, amsthm, amsfonts}
\usepackage{geometry}
\usepackage{comment}
\usepackage{enumitem}
 \usepackage{amsrefs}
\usepackage[official]{eurosym}
\usepackage{graphicx}
\usepackage{nomencl}
\usepackage{bm}

\usepackage{amsfonts}
\usepackage{hyperref}
\usepackage{amssymb}
\usepackage{fancyhdr}
\usepackage{amscd}
\usepackage[english]{babel}

\theoremstyle{plain}
\theoremstyle{definition}

\newtheorem*{theorem*}{Theorem}
\newtheorem*{acknowledgement}{Acknowledgements}
\newtheorem{thm}{Theorem}
\newtheorem{pr}[thm]{Proposition}
\newtheorem{cor}[thm]{Corollary}
\newtheorem{lem}[thm]{Lemma}
\newtheorem{theorem}{Theorem}

\theoremstyle{definition}
\newtheorem{dfn}[thm]{\scshape{Definition}}

\newcommand\T{{\mathcal{T}}}
\makeatletter
\def\cleardoublepage{\clearpage\if@twoside \ifodd\c@page\else
	\hbox{}
	\thispagestyle{empty}
	\newpage
	\if@twocolumn\hbox{}\newpage\fi\fi\fi}
\makeatother
\DeclareMathOperator{\Aut}{Aut}
\DeclareMathOperator{\St}{st}

\def\l{\langle}
\def\r{\rangle}
\def\o{\overline}

\def\ov{\overset}
\def\mc{\mathcal}
\def\c{\curvearrowright}

\keywords{Fractal groups, Engel elements}

\begin{document}

\title[Engel elements in some fractal groups]{Engel elements in some fractal groups}
\author[G.A. Fern\'andez-Alcober]{Gustavo A. Fern\'andez-Alcober}
\address{Department of Mathematics\\ University of the Basque Country UPV/EHU\\
48080 Bilbao, Spain}
\email{gustavo.fernandez@ehu.eus}
\author[A. Garreta]{Albert Garreta}
\address{Department of Mathematics\\ University of the Basque Country UPV/EHU\\
48080 Bilbao, Spain}
\email{albert.garreta@ehu.eus}
\author[M. Noce]{Marialaura Noce}
\address{Dipartimento di Matematica\\ Universit\`a di Salerno\\
Via Giovanni Paolo II, 132\\ 84084 Fisciano (SA)\\ Italy -- Department of Mathematics\\ University of the Basque Country UPV/EHU\\
48080 Bilbao, Spain }
\email{mnoce@unisa.it}

\thanks{All three authors are supported by the Spanish Government grant MTM2017-86802-P and by the Basque Government grant IT974-16.
The first author is also supported by the Spanish Government grant MTM2014-53810-C2-2-P, the second author by the ERC grant 336983, and the third author by the ``National Group for Algebraic and Geometric Structures, and their Applications'' (GNSAGA \-- INdAM)}

\begin{abstract}
%
Let $p$ be a prime and let $G$ be a subgroup of a Sylow pro-$p$ subgroup of the group of automorphisms of the $p$-adic tree. We prove that if $G$ is fractal and $|G':\St_G(1)'|=\infty$, then the set $L(G)$ of left Engel elements of $G$ is trivial. 
%
This result applies to fractal nonabelian groups with torsion-free abelianization, for example the Basilica group, the Brunner-Sidki-Vieira group, and also to the GGS-group with constant defining vector.

We further provide two examples showing that neither  of the requirements $|G':\St_G(1)'|=\infty$ and being fractal can be dropped. 
%
%
\end{abstract}

\maketitle


\section{Introduction}

In this paper we prove that certain fractal groups have no nontrivial  Engel elements. These include the Basilica group, the Brunner-Sidki-Vieira group, and  the GGS-group with constant defining vector. The complete family is described in \cref{marialaura_intro}. 

An element $g$ of a group $G$ is called \emph{left Engel} if for every $x \in G$ there exists an integer $n = n(g, x) \geq 1$ such that $[x, g, \overset{n}{\dots}, g] = 1$. The set of all left Engel elements of $G$ is denoted by $L(G)$, and if $L(G)=G$ we say that the group  $G$ is an \emph{Engel group}. Major areas of interest in Engel theory stem from the following questions: 
\begin{enumerate}
\item If $G$ is an Engel group, is $G$ locally nilpotent?
\item Is $L(G)$  a subgroup of $G$?
\end{enumerate}
Question 1 was considered by Burnside in the same paper of 1902 where he posed the famous Burnside problem. In general, the answer to Question 1 is negative:  the Golod-Safarevich  groups  are Engel, but not locally nilpotent.  The latter are the only known finitely generated example  with this property.  On the other hand, Question 1 has a positive answer for some  families of groups such as: finite groups (Zorn, 1936),  groups that have the max property (Baer, 1957), soluble groups (Gruenberg, 1959), and compact groups (Medvedev, 2003).

Notice that some Golod-Safarevich groups provide also a negative answer to the Burnside problem, i.e.\ they are infinite finitely generated groups all of whose elements have finite order. Other examples of groups with these properties are given by many fractal branch groups such as the first Grigorchuk group $\mathfrak{G}$ and the Gupta-Sidki group $G$. One may ask if these groups constitute a negative answer to the ``Engel problem'' as well (i.e.\ to Question 1). Bartholdi \cite{Bartholdi} showed in  2016 that this is not the case for $\mathfrak{G}$ and  $G$.  More precisely, he proved that $L(\mathfrak{G}) = \{ g\in \mathfrak{G} \mid g^2 =1\}$ and $L(G)=1$.
%
%
%
%
%
Observe that Bartholdi's result for $\mathfrak{G}$ answers Question 2 negatively. We remark that Bludov  provided a similar (unpublished) example in 2005 (see \cite{ml}). It is also noticeable that there exists an automaton group where determining if an element is left Engel is undecidable \cite{Gillibert}.

Likewise considerations can be made for right Engel elements and bounded right or left Engel elements (see the end of this introduction). We refer to the surveys \cites{abdollahi, Gunnar} for further general information. See also  \cites{marialaura300, MT} for  recent results on right and bounded Engel elements in some types of branch groups.



The paper is organized as follows. In Section 2, we introduce some basic terminology regarding (strongly) fractal groups $G$ of automorphisms of the $d$-adic tree. We prove that, for such groups,  $L(G)\subseteq \St_G(1)$ implies that $L(G)=1$.

In Section 3, we consider group actions $G\c H$. In this context, we use a more general notion of commutator. This is achieved by defining $[h,g]=h^{-1} h^g$, where $h^g$ is the action of $g$ on $h$, and  $h\in H$, $g\in G$. One then denotes by $L(G \c H)$ the set of all $g\in G$ such that for all $h\in H$, we have $[h,g, \overset{n}{\dots}, g]=1$ for some $n$. 

We start by showing that if $F\c A$ is an Engel action of a finite group $F$ on a finitely generated abelian group $A$, then  $\langle [A,F] \rangle$ is  finite. 
%
%
Then, we prove the main result of the paper, which we state below. By $\Gamma$ we denote the standard Sylow pro-$p$ subgroup of $\Aut\T_p$ (see \cref{l777}). 
\begin{theorem}\label{marialaura_intro}
Let $G \leq \Gamma$ be a fractal group such that $|G':\St_G(1)'| = \infty$. Then $L(G) = 1$.
\end{theorem}
The main idea of the proof is the following:  denote by $\bar{\ }$ the projection of $G$ onto $G/\St_G(1)'$. The conjugation  $G\c \St_G(1)$ induces an action $\o G \c\o{\St_G(1)}$. Since $\o{\St_G(1)}$ is abelian, the latter yields an action  
\begin{equation}\label{p}
\o G/ \o{\St_G(1)}\c \o{\St_G(1)}.
\end{equation} 
Note that $\o G/ \o{\St_G(1)}\cong C_p$. Using our previous results, we conclude that either the action \eqref{p} is Engel and then $|G': \St_G(1)'|< \infty$, or else it is trivial and  $L(G)\subseteq \St_G(1)$, in which case  $L(G)=1$.

In Section 4, we apply \cref{marialaura_intro} to a certain class of fractal groups, and we obtain the following results.
\begin{theorem}\label{marialaura_cor_intro}
Let $G\leq \Gamma$ be a  fractal  group with  torsion-free abelianization. Then $L(G)=1$.
\end{theorem}
\begin{theorem}
The following groups have no nontrivial left Engel elements: the Basilica group, the Brunner-Sidki-Vieira group, and the GGS-group with constant defining vector $\mc{G}$.
\end{theorem}
\cref{marialaura_cor_intro} can be applied directly to the Basilica and the Brunner-Sidki-Vieira group. This is not the case with $\mc{G}$. Instead, we reduce the study of $L(\mc{G})$ to a subgroup $K$ that meets (up to conjugation in $\Aut\T_p$) the requirements of \cref{marialaura_cor_intro}.

In Section 5 we finish the paper by showing that none of the conditions $|G':\St_G(1)'|=\infty$ and being fractal can be omitted in \cref{marialaura_intro}.
%


%

We end this introduction by pointing out that if  $L(G)=1$ for some group $G$, then  $\o L(G)= R(G)= \o R(G)=1$. This is due to the following inclusions (proved in \cite{Hein}):
%
%
$
R(G)^{-1}, \o L(G), \o R(G)^{-1} \subseteq L(G).
$
%
%
Here, $R(G)$ stands for the set of right Engel elements,  and the bars indicate that the elements are bounded left or right Engel, respectively (see \cites{abdollahi,Gunnar}).
%





\addtocontents{toc}{\protect\setcounter{tocdepth}{1}}
    
\section{Engel elements in strongly fractal groups}

Throughout the paper we will use the following standard notation:  $[x,_n y]=[x, y, \ov{n}{\dots} y]= [[x,_{n-1},y],y]$ for $n\geq 2$, and $[x,_1 y]=[x,y]=x^{-1}y^{-1}xy$. We also let $[x,_0 y]=x$.

We start this section by fixing some terminology regarding groups $G$ of automorphisms of rooted $d$-adic trees (see \cite{branch} or \cites{nekrashevych} for a detailed discussion on this topic). Afterwards, we  prove that if $G$ is strongly fractal, then $L(G)\subseteq \St_G(1)$ implies that $L(G)=1$ (\cref{ts}).  

%

Let $d$ be a positive integer, and $\T_d$ (or $\T$)  the rooted $d$-adic tree. The group of automorphisms of $\T$ is denoted by $\Aut\T$. We let $L_n$ be the $n$-th level of $\T$. The $n$-th \emph{level stabilizer} of $\Aut\T$ is the subgroup of $\Aut\T$ that fixes every vertex of $L_n$. We will denote it by  $\St(n)$.  More generally, if $G \leq \Aut\T$, we define the $n$-th level stabilizer of $G$ as $\St_{G}(n) = \St(n) \cap G$.
Notice that these stabilizers are normal subgroups of the corresponding group. We let $\psi$ be the isomorphism
\begin{align*}
\psi: \St(1) & \longrightarrow \Aut\T \times \overset{d}{\cdots} \times \Aut\T \label{psi_iso}\\
g & \longmapsto (g_{u})_{u \in L_1}, \nonumber
\end{align*}
where $g_{u}$ is the section of $g$ at the vertex $u$, i.e.\ the action of $g$ on the subtree $T_u$ that hangs from the vertex $u$. An automorphism $s \in \Aut\T$ is called \emph{rooted} if there exists a permutation $\sigma \in S_{d}$ such that $s$ rigidly permutes the trees $\{T_u \mid u\in L_1\}$  according to the permutation $\sigma$. We will identify $s$ and $\sigma$.
Notice that if $g \in \St(1)$ with $\psi(g)=(g_1, \dots, g_d)$, and $\sigma$ is a rooted automorphism, then,
\begin{equation}\label{carmine}
\psi(g^{\sigma})= \left(g_{\sigma^{-1}(1)},  \dots, g_{\sigma^{-1}(d)}\right).
\end{equation}
Any element $g$ of a group of automorphisms $G\leq \Aut\T$ can be written uniquely in the form $g= g_0 \sigma$, where $g_0\in \St(1)$ and $\sigma$ is a rooted automorphism. The group $G\leq \Aut\T$ is said to be \emph{self-similar} if $$\psi(g_0) \in  G \times \overset{d}{\cdots} \times G$$ for all $g\in G$ (thus for example $\Aut\T$ is self-similar).

The decomposition $g=g_0 \sigma$ above  applied to the elements of $\Aut\T_d$, together with the action \eqref{carmine}, yields  isomorphisms
\begin{align}
\begin{split}\label{e: iterated_wreath}
\Aut\T_d &\cong \St(1) \rtimes S_d \cong \left(\Aut\T_d \times \ov{d}{\dots} \times \Aut\T_d\right) \rtimes S_d \\ 
&\cong \Aut\T_d \wr S_d \cong (( \dots \wr S_d) \wr S_d) \wr S_d. 
\end{split}
\end{align}
\begin{dfn}\label{l777}
Let $p$ be a prime and let $\sigma$ be the rooted automorphism of $\Aut\T_p$ corresponding to the cycle $(1 \dots p)$. The \emph{Sylow pro-$p$ subgroup} of $\Aut\T_p$ induced by $\sigma$ is the subgroup $\Gamma\leq\Aut\T_p$ that is mapped to $ (( \dots \wr \langle \sigma \rangle) \wr \langle \sigma \rangle) \wr \langle \sigma \rangle = (( \dots \wr  C_p) \wr C_p ) \wr  C_p$ under the isomorphism \eqref{e: iterated_wreath}. Note that if $p=2$ we have $\Gamma = \Aut\T_2$.
\end{dfn}

In this paper we will mostly study self-similar groups $G\leq \Aut\T_p$  such that $G\leq \Gamma$. In this case, every element $g \in G$ can be written in the form $g=h \sigma^t$, for some $t\in \mathbb{Z}$ and $h \in \St(1)$ such that $\psi(h) \in G\times \dots \times G$. 

A self-similar group $G$ is called \emph{strongly fractal} if $\pi_i(\St_G(1))=G$ for all $i =1, \dots, d$, where $\pi_i$ is the projection onto the $i$-th component of $G \times \ov{d}{\dots} \times G$.
If $G \leq \Gamma$, then this is equivalent to saying that $G$ is fractal,  by Lemma 2.5 of \cite{Jone}.  



We now  study the set of left Engel elements of strongly fractal groups $G\leq \Aut\T_d$.  
We will use the following identity
\begin{equation}\label{zapatillas_nuevas}
\psi\left([h,_n g]\right)=\left( [h_1,_n g_1], \dots, [h_d,_n g_d] \right),
\end{equation}
which holds for any two elements $h, g\in \St(1)$. Here 
the 
$h_i$'s and  $g_i$'s are the components of $\psi(g)$ and $\psi(h)$.

\begin{lem}\label{mrs_robinson}
Let $G \leq \Aut\T_d$ be a strongly fractal group.  Then
\begin{equation*}
L(G)\cap \St_G(1) = \left\{ h\in \St_G(1) \mid \psi(h) \in L(G) \times \overset{d}{\dots} \times L(G) \right\}.
\end{equation*}
%
%
\end{lem} 
\begin{proof}
Let $g\in L(G)\cap \St_G(1)$, and write $\psi(g)=(g_1, \dots, g_d)$. We show that each $g_i$ is left Engel. 
Assume for contradiction that there exists $h \in G$ be such that $[h,_n g_i]\neq 1$ for all $n$. Since $G$ is strongly fractal, there exists an element $s \in \St_G(1)$ such that $\psi(s)=(h_1, \dots, h_{i-1}, h, h_{i+1}, \dots, h_d)$, for some   $h_1, \dots, h_{i-1}, h_{i+1}, \dots, h_d \in G$. Then, by \eqref{zapatillas_nuevas},
$$
\psi\left([s,_n g]\right)=\left([h_1,_n g_1], \dots, [h_{i-1},_n g_{i-1}], [h,_n g_i],[h_{i+1},_n g_{i+1}], \dots, [h_d,_n g_d] \right).
$$
This element is nontrivial because $[h,_n g_i]\neq 1$ for all $n$, and so $[s,_n g] \neq 1$ for all $n$, contradicting the fact that $g \in L(G)$. 

To prove the reverse inclusion, let $h \in \St_G(1)$ be such that all components of $\psi(h)$ are left Engel, and let $g$ be any element of $G$. Then, since $\St_G(1)$ is normal in $G$, we have $[g,h]\in \St_G(1)$. It follows now from the equality \eqref{zapatillas_nuevas} that $[g,_n h]=1$ for some $n$, as required.
\end{proof}


The following constitutes a key observation.
\begin{pr}\label{ra}
Let $S$ be a subset of $\Aut\T_d$. Suppose that $S\subseteq \St(1)$ and that $\psi(S)\subseteq S \times \ov{d}{\dots} \times S$. Then $S=1$.
\end{pr}
\begin{proof}
Suppose that $S\neq 1$. Then there exists a maximum $n$ such that $S \subseteq \St(n)$, but $S \not \subseteq \St(n+1)$. Let $s\in S$. In particular, $s\in \St(1)$, and the components $s_1, \dots, s_d$ of  $\psi(s)$ belong to $S$ by hypothesis. Since  $S\subseteq \St(n)$, each $s_i$ stabilizes the first $n$ levels of $\T_d$, and hence $s \in \St(n+1)$. This occurs with all $s\in S$, yielding the contradictory inclusion $S\subseteq \St(n+1)$. Therefore $S=1$. 
\end{proof}

The following is an immediate consequence of \cref{mrs_robinson} and \cref{ra}.

\begin{cor}\label{ts}
Let $G \leq \Aut\T_d$ be a strongly fractal group such that $L(G) \subseteq \St_G(1)$. Then $L(G)=1$.
\end{cor}

\section{Engel elements, group actions, and fractal groups}
\addtocontents{toc}{\protect\setcounter{tocdepth}{3}}\label{obm}

In this section, we generalize the notion of commutator and of left Engel element  to the context of group actions $G\c H$. 
We will use exponential notation to refer to the action of a group $G$ on another group $H$, so that $h \cdot g = h^g$ for all $g\in G, h\in H$.

Given an action of groups $G\c H$, one may define commutators by letting $[h,g]=h^{-1}h^{g}$, for $g\in G$ and $h\in H$. Then, $[h,_n g]$ is defined similarly as before. Of course, if $H$ is a normal subgroup of $G$ and the action $G\c H$ is the standard conjugation, then $[h,g]$ is the usual commutator $h^{-1}g^{-1}h g$. An element $g\in G$ is called \emph{left Engel with respect to $G\c H$} if for all $h\in H$  there exists $n\geq 1$ such that $[h,_n g]=1$. The set of left Engel elements of $G$ with respect to $G\c H$ will be denoted $L(G\c H)$. If $L(G\c H)=G$, then  $G\c H$ is called an \emph{Engel action}. Given $S \subseteq H$, and $T \subseteq G$, we denote $\langle [S, T] \rangle =  \langle  [s, t] \mid s \in S, t \in T \rangle$.


We start by proving a key fact regarding  how periodic groups can act on finitely generated abelian groups. 

\begin{pr}
\label{gustavobis}
Let $G\c A$ be an Engel action of a finite group $G$ on a finitely generated  abelian group $A$. Then $\l [A, G] \r$ is finite. As a consequence, if $A$ is free abelian, then the action $G\c A$ is trivial.
\end{pr}
\begin{proof}
Let $\ell$ be the order of $G$, and take two elements $a \in A$, and $g \in G$. We claim that if  $[a, _n g] = 1$ for some $n$, then $[a, g]^{\ell^{n-1}} = 1$. We argue by induction on $n$, the case $n =1$ being obvious. Let $n>1$, and suppose $[a, _{n} g] = 1$. If we denote $s=[a,_{n-2} g]$, we have $[s, g, g]=1$. One can then prove by induction that $[s,g^k]= [s,g]^k$ for all $k\in \mathbb{Z}$. Then $[a,_{n-1} g]^{\ell}=[s,g]^{\ell}=[s,g^{\ell}]=1$. 
Write $K = \l [a,_{n-1} g]^h \mid h \in G \r$, $\o{A} = A/K$, and consider the Engel action $G \c \o{A}$, which is still an Engel action of a finite group on a finitely generated abelian group. Then, $[\bar{a}, _{n-1} g] = \o{[a, _{n-1} g]} = \o{1}$ and, by induction, $[\bar{a}, g]^{\ell^{n-2}} = \o{1}$. Thus, $[a, g]^{\ell^{n-2}} \in K$, and so $[a, g]^{\ell^{n-1}} = 1$ because $K$ is abelian generated by elements of order dividing $\ell$. This completes the proof of the claim.

We have proved that each element of the generator set $[A, G]$ of $\l [A, G] \r$ has finite order.
Since $\l [A, G] \r$ is finitely generated and abelian, we conclude that $\l [A, G]\r$ is finite.
The last part of the lemma immediately follows. 
\end{proof}

%
%


%
%

%
%
%

Now we can proceed to the proof of our main theorem.

\begin{thm}
\label{albert}
Let $G \leq \Gamma$ be a fractal group such that $|G':\St_G(1)'| = \infty$.
Then $L(G) = 1$.
\end{thm}

\begin{proof}
Write $S$ for $\St_G(1)$.
We claim that $L(G)\subseteq S$, from which $L(G)=1$ follows by using Corollary \ref{ts} and the fact that subgroups of $\Gamma$ are strongly fractal if and only if they are fractal (Lemma 2.5 of \cite{Jone}).
By way of contradiction, we assume that $L(G)\not\subseteq S$.
Since $G\le \Gamma$, this implies that $|G:S|=p$ and also the factor group $G/S'$ is soluble.
By a result of Gruenberg (Theorems 2 and 4 of \cite{Gruenberg2}), $L(G/S')$ is a subgroup of $G/S'$.

Since $S/S'$ is an abelian normal subgroup of $G/S'$, we have $S/S'\subseteq L(G/S')$ and then either $L(G/S')=S/S'$ or $G/S'$.
In the former case, we have $L(G)\subseteq S$, which is a contradiction.
In the latter, $G/S'$ is an Engel group.
Since $S/S'$ is abelian, the action of $G$ on $S$ by conjugation induces an action of $G/S$ on $S/S'$, by which we have $[sS',gS]=[s,g]S'$ for all $s\in S$ and $g\in G$. 
Thus this action is Engel.
By Proposition \ref{gustavobis}, the subgroup $\langle [G/S,S/S'] \rangle$ is finite, i.e.\
$\langle [G,S] \rangle/S'$ is finite.
Now since $G/S$ is cyclic, it follows that $G'=\langle [G,S] \rangle$ and we conclude that $G'/S'$ is finite.
This contradiction completes the proof of the theorem.
\end{proof}

\section{Applications to specific fractal groups}

In this section we apply our main result to a family of fractal groups.
Since $R(G)^{-1}, \o L(G), \o R^{-1}(G) \subseteq L(G)$ for any group $G$, by proving that $L(G)=1$ one automatically obtains that each one of these sets is also trivial. We will omit this observation in the statement of the subsequent results.


\begin{thm}\label{marialaura}
Let $G\leq \Gamma$ be a nonabelian fractal group with torsion-free abelianization.  Then $L(G)=1$.
\end{thm}
\begin{proof}
Denote $S=\St_G(1)$. By \cref{albert}, it suffices to prove that $|G':S'|= \infty$. Suppose towards contradiction that this is not the case. Let $\pi_i$ be the projection  of  $G \times \ov{p}{\dots} \times G$ onto its $i$-th component. Notice the chain of inclusions $S' \leq G' \leq S\leq G$.  Since $G$ is strongly fractal, $\pi_i(\psi(S')) = G'$ for all $i=1, \dots, p$.
By assumption $|G':S'|< \infty$, hence   $|\psi(G') : \psi(S')| < \infty$ and $|\pi_i(\psi(G')) : \pi_i(\psi(S'))| < \infty$, for all $i=1, \dots, p$. Thus, $\pi_i(\psi(G'))/G'$ is a finite subgroup of $G/G'$.  Since $G/G'$ is torsion-free, it follows that  $\pi_i(\psi(G')) = G'$  for all $i$. This implies that $\psi(G') \leq G' \times \ov{p}{\dots} \times G'$. In this case, by \cref{ra}, we have $G'=1$, which is a contradiction. 
%
%
\end{proof}




\subsection{The Basilica group and the Brunner-Sidki-Vieira group}


The \emph{Basilica group} $\mc{B}$ is the subgroup of $\Aut\T_2$ generated by the two automorphisms $a$ and $b$ defined as:
\begin{equation*}
a = (1, b), \quad \quad b = (1, a)\sigma,
\end{equation*}
where $\sigma$  is the rooted automorphism of $\T_2$ that corresponds to the permutation $(12)$. Here we make the abuse of notation of writing $a=(1,b)$ instead of $\psi(a)=(1,b)$, and similarly for the element $(1, a)$.

The \emph{Brunner-Sidki-Vieira group}  $\mc{S}$ is the group of automorphisms of $\T_2$ generated by
\begin{equation*}
c = (1, c^{-1})\sigma, \quad \quad d = (1, d)\sigma,
\end{equation*}
where $\sigma$ is the rooted automorphism of $\T_2$ corresponding to the cycle $(12)$. Again, we have omitted the map $\psi$ in the definition of $c$ and $d$.

In \cite{ZUK} and \cite{BRUNNER} it is proved that $\mc{B}$ and $\mc{S}$ have torsion-free abelianization. Then, the following holds.

\begin{thm}
The Basilica group and the Brunner-Sidki-Vieira group have no nontrivial left Engel elements.
\end{thm}
%
%
%

\subsection{The GGS-group with constant defining vector}

Let $p$ be an odd prime, and $\bm{e} = (e_1, \dots, e_{p-1})$ a vector such that $e_i\in \{0,\dots, p-1\}$ and such that not all the $e_i$ are $0$. The GGS-group $\mc{G}_{\bm{e}}$, named after Grigorchuk, Gupta, and Sidki, is the group generated by the two automorphisms $a, b\in \Aut\T_p$, where $a$ is the rooted automorphism corresponding to the cycle $(1 \dots p)$, and $b\in \St(1)$ is:
$$
b=(a^{e_1}, \dots, a^{e_{p-1}}, b).
$$
The group $\mc{G}_{\bm{e}}$ is fractal for any vector $\bm{e}$, but other properties depend on the choice of $\bm{e}$ (see \cite{Gustavo}). For instance, $\mc{G}_{\bm{e}}$ is just infinite if and only if $\bm{e}$ is not constant. In this paper we  consider only the case when $\bm{e}$ is constant:\ i.e.\ $e_1= \dots = e_{p-1}=n$ for some nonzero $n$.  Since proportional nonzero vectors define the same GGS-group, we may assume that $n=1$. Throughout this section we let $\mc{G}$ denote $\mc{G}_{(1, \dots, 1)}$. 

According to Lemma 4.2 of \cite{Gustavo}, $\mc{G}$  has a normal subgroup $K$ of index $p$ such that $\mc{G}$ is weakly regular branch over $K'$.
Moreover, if $\bar{\ }$ denotes the projection $\mc{G} \to \mc{G}/K'$, then by Proposition 3.4 of \cite{GustavoVonDyck}, we have $\o{\mc{G}}=\o K\ltimes \langle \o a \rangle$, with
$\o K \cong C_{\infty} \times \overset{p-1}{\dots}\times C_{\infty}$, and $\langle \o a \rangle\cong C_p$. 

%
%
%

Our goal is to prove that $L(\mc{G})=1$. 
We first show that $L(\mc{G}) \subseteq L(K)$ by studying $L(\o{\mc{G}})$, and then  we prove that $L(K)=1$.

\begin{lem}\label{mm6}
We have $L(\o{\mc{G}})=\o K$. As a consequence, $L(\mc{G})\subseteq L(K)$.
\end{lem}
\begin{proof}
Since $\o{\mc{G}}$ is soluble, $L(\o{\mc{G}})$ is a subgroup of $\o{\mc{G}}$, and since $\o K$ is abelian, we have $\o K\le L(\o{\mc{G}})$.
Thus either $L(\o{\mc{G}})=\o K$ or $\o{\mc{G}}$.
In the latter case, the action of $\langle \o a \rangle$ on $\o K$ is Engel and, since $\o K$ is free abelian, this action is trivial, by \cref{gustavobis}.
Hence $\o{\mc{G}}$ is abelian and $\mc{G}'=K'$, which is a contradiction, since $|\mc{G}:\mc{G}'|=p^2$ is finite by Theorem 2.1 of \cite{Gustavo}.
It follows that $L(\o{\mc{G}})= \o K$ and, since $\o{L(\mc{G})}\subseteq L(\o{\mc{G}})$, we get
$L(\mc{G})\subseteq L(K)$.
%
%
%
%
%
%
%
%
%
%
%
%
\end{proof}

We now proceed to prove that $L(K)=1$. Unfortunately, $K$ is not self-similar. This can be fixed by appropriately conjugating $K$  in $\Aut\T_p$.

\begin{lem}[\cite{AleJone}]\label{AJ}
Let $h\in \St(1)$ be such that $\psi(h)=(ah,a^2h,\dots, a^{p-1}h, h)$. Then $K^h$ is strongly fractal.  
\end{lem}
\begin{proof}
As shown in Lemma 4.2 of \cite{Gustavo},  $K$ is generated by $y_0= ba^{-1}$, and $y_i= y_0^{a^i}$ ($i=1, \dots, p-1$). Write 
\begin{align*}
z_1=(z_1, 1, \dots, 1)a^{-1},\
z_2=(1, z_2, 1, \dots, 1)a^{-1},\
 \dots,\
z_p=(1, \dots, 1, z_p)a^{-1}.
\end{align*} By making computations, one may check that $y_i^{h}=z_{i}$ for all $i$ (reading the subindices modulo $p$). Hence, $K^h$ is generated by the $z_i$. It is now clear that $K^h$ is  self-similar,   and strong-fractalness of $K^h$ is a consequence of the identity $z_i^p=(z_i, \dots, z_i)$, which holds for all $i$.
\end{proof}
%
%

\begin{thm}
The GGS-group with constant defining vector has no nontrivial Engel elements.
%
\end{thm}
\begin{proof}
By \cref{mm6}, it suffices to show that $L(K)=1$. Let $h$ be the element of \cref{AJ}. Since $K\cong K^h$, one has that $L(K)=1$ if and only if $L(K^h)=1$. To prove the latter  we will use \cref{marialaura}.   \cref{AJ} states that $K^h$ is fractal, and clearly $K^h$ is a subgroup of the Sylow pro-$p$ subgroup $\Gamma$.
On the other hand, $(K^h)/(K^h)' \cong K/K'$ is torsion-free.
Hence, by \cref{marialaura}, we have $L(K^h)=1$, and we conclude that $L(\mc{G})=1$.
\end{proof}

\section{The lamplighter group and the adding machine:  examples}

\cref{albert} states that $L(G)=1$ for any fractal group $G$ such that $G \leq \Gamma\leq \Aut\T_p$ and $|G':\St_G(1)'|=\infty$. In this section we show that if $|G:G'|=\infty$ and $|G':\St_G(1)'|<\infty$ this is no longer true. We also prove that the condition of being fractal cannot be dropped.

Recall that the lamplighter group $\mc{L}$ is the metabelian group $C_2\wr C_\infty$. It is well known that $\mc{L}$ can also be seen as the group of automorphisms of the binary tree $\T_2$  generated 
by $a=(a,a\sigma)$ and  the rooted automorphism $\sigma$ corresponding to the cycle $(12)$.


%
\begin{pr}\label{mandarina}
The lamplighter group $\mc{L}$ satisfies the following properties: $\mc{L}\leq \Gamma=\Aut\T_2$, $|\mc{L}:\mc{L}'|= \infty$, $\mc{L}$ is fractal, $|\mc{L}' : \St_{\mc{L}}(1)'|$ is finite, and $L(\mc{L}) \neq 1.$
\end{pr}

\begin{proof}
The first property is trivially satisfied. To see that the group $\mc{L}$ is fractal, notice that the first components of $a$ and $a^{\sigma}$ generate $\mc{L}$, since $a=(a, a\sigma)$ and $a^{\sigma}=(a\sigma, a)$. The same holds  for the second components. It is also well known that the abelianization of $\mc{L}$ is isomorphic to $C_2 \times C_{\infty}$ 
 with $\l \sigma \mc{L}'\r \cong C_2$ and $\l a \mc{L}' \r \cong C_{\infty}$.

We now prove that the index of  $\St_{\mc{L}}(1)'$ in $\mc{L}'$ is finite. Let us write $S = \St_{\mc{L}}(1)$. Notice that $S' \subseteq \mc{L}' \subseteq S \subseteq \mc{L}$.  One can compute that $S = \l a, a^{\sigma} \r$.  
Now, letting $c=[a,\sigma]$,
\begin{align}\label{stablamp}
 S/S' = \l aS', a^{\sigma}S' \r = \l aS', acS'\r = \l aS', cS'\r.   
\end{align}
%
We claim that $\mc{L}'/S' = \l cS'\r$. Indeed, let $y \in \mc{L}'/S'$. In particular,  $y \in S/S'$, and by \eqref{stablamp}, $y = a^{n}c^{m}S'$, for some $m$ and $n$. Then, $a^{n}S' \in \mc{L}'/S'$ and thus $n=0$, because $\langle a\mc{L'} \rangle \cong C_{\infty}$.  It follows that $\mc{L}'/S' \leq \l cS'\r$, and so $\mc{L}'/S' = \l cS'\r$. Notice that $c$ has order 2 because $c= [a, \sigma] = a^{-1}a^{\sigma} = (\sigma, \sigma)$. Then, $|\mc{L}': S'| \leq 2$ (in fact, $|\mc{L}': S'| =2$ by \cref{ra}).

Finally, since the base group of $\mc{L}$ is abelian and normal, it is contained in $L(\mc{L})$ and
$L(\mc{L})\ne 1$.
%
%
%
\end{proof}

We next show that the requirement of being fractal is necessary in \cref{albert}. 
Let  $H$ be the subgroup of $\Aut\T_2$ generated by $\sigma$ and $x$, where $\sigma$ is again the rooted automorphism corresponding to the cycle $(12)$, and $x=(1,x)\sigma$ is the so-called adding machine.

\begin{pr}
The group $H\leq \Gamma=\Aut\T_2$ is not fractal, $|H': \St_H(1)'|=\infty$, and $L(H)=\St_H(1)$. 
\end{pr}
\begin{proof}
Let  $b=x\sigma$ and note that $b \in \St(1)$ and $b^2=(x,x)$. Then both elements $x$ and $b$ have infinite order. 
By easy computations, one can see that  $\St_H(1) = \langle b, b^\sigma \rangle$. Since $b=(1,x)$ and $b^\sigma=(x,1)$, we have $\St_H(1) \cong C_{\infty} \times C_{\infty}$. In particular, $\St_H(1)'=1$.  Moreover $H' = \langle [\sigma,b] \rangle^{H}= \langle (x^{-1}, x)\rangle^H = \{(x^{\pm n}, x^{\mp n}) \mid n\in \mathbb{Z}\}$, and it follows that $|H':\St_H(1)'|=\infty$. 
Note also that $H$ is not fractal because $b=(1,x)$ and $b^\sigma=(x,1)$. 

We now prove that $L(H)=\St_H(1)$.
One inclusion is obvious, since $\St_H(1)$ is abelian and normal in $H$.
Since $H$ is soluble, it follows that $L(H)=\St_H(1)$ or $H$.
In the latter case, the action of $H/\St_H(1)$ on $\St_H(1)$ is Engel and, by \cref{gustavobis}, this action must be trivial.
This implies that $H'=[H,\St_H(1)]=1$, which is a contradiction.
\end{proof}

\begin{acknowledgement}
We thank Alejandra Garrido and Jone Uria-Albizuri for communicating Lemma \ref{AJ} to us.
\end{acknowledgement}

\bibliography{bib}

\begin{bibdiv}
\begin{biblist}

\bib{abdollahi}{incollection}{
      author={Abdollahi, A.},
       title={Engel elements in groups},
        date={2011},
   booktitle={Groups {S}t {A}ndrews 2009 in {B}ath, {V}olume 1, {C}ambridge
  {U}niversity {P}ress},
      series={London Math. Soc. Lecture Note Ser.},
   publisher={Cambridge University Press, Cambridge},
       pages={94\ndash 117},
}

\bib{Bartholdi}{inproceedings}{
      author={Bartholdi, L.},
       title={Algorithmic decidability of {E}ngel's property for automaton
  groups},
        date={2016},
   booktitle={Computer {S}cience -- {T}heory and {A}pplications, {L}ecture
  {N}otes in {C}omputer {S}cience, {V}olume 9691, {S}pringer},
      editor={Kulikov, Alexander~S.},
      editor={Woeginger, Gerhard~J.},
   publisher={Springer},
     address={Cham},
       pages={29\ndash 40},
}

\bib{branch}{incollection}{
      author={{Bartholdi}, L.},
      author={{Grigorchuk}, R.~I.},
      author={{Sunik}, Z.},
       title={Branch groups},
        date={2003},
   booktitle={Handbook of {A}lgebra, {V}olume 3, {N}orth-{H}olland},
       pages={989\ndash 1112},
}

\bib{BRUNNER}{article}{
      author={Brunner, A.~M.},
      author={Sidki, S.},
      author={Vieira, A.~C.},
       title={A just-nonsolvable torsion-free group defined on the binary
  tree},
        date={1999},
        ISSN={0021-8693},
     journal={Journal of Algebra},
      volume={211},
      number={1},
       pages={99\ndash 114},
  url={http://www.sciencedirect.com/science/article/pii/S0021869398975792},
}

\bib{GustavoVonDyck}{article}{
      author={{Fern{\'a}ndez-Alcober}, G.~A.},
      author={{Garrido}, A.},
      author={{Uria-Albizuri}, J.},
       title={{On the congruence subgroup property for GGS-groups}},
        date={2017},
     journal={Proceedings of the American Mathematical Society},
      volume={145},
       pages={3311\ndash 3322},
}

\bib{Gustavo}{article}{
      author={Fern{\'a}ndez-Alcober, G.~A.},
      author={Zugadi-Reizabal, A.},
       title={{GGS-groups: order of congruence quotients and Hausdorff
  dimension}},
        date={2013},
     journal={Transactions of the American Mathematical Society},
      volume={366},
       pages={1993\ndash 2017},
}

\bib{marialaura300}{article}{
      author={Fern\'andez-Alcober, G.A.},
      author={Noce, M.},
      author={Tortora, A.},
       title={{Engel} elements in strongly fractal groups},
         how={to appear},
        date={2017},
     journal={preprint},
}

\bib{AleJone}{misc}{
      author={Garrido, A.},
      author={Uria-Albizuri, J.},
       title={\emph{private communication}},
         how={private communication},
}

\bib{Gillibert}{article}{
      author={Gillibert, P.},
       title={An automaton group with undecidable order and {E}ngel problems},
        date={2018},
        ISSN={0021-8693},
     journal={Journal of Algebra},
      volume={497},
       pages={363\ndash 392},
  url={http://www.sciencedirect.com/science/article/pii/S0021869317306579},
}

\bib{ZUK}{article}{
      author={Grigorchuk, R.~I.},
      author={Zuk, A.},
       title={On a torsion-free weakly branch group defined by a three state
  automaton},
        date={2002},
     journal={International Journal of Algebra and Computation},
      volume={12},
       pages={223\ndash 245},
}

\bib{Gruenberg2}{article}{
      author={Gruenberg, K.~W.},
       title={The {Engel} elements of a soluble group},
        date={1959},
     journal={Illinois Journal of Mathematics},
      volume={3},
      number={2},
       pages={151\ndash 168},
         url={https://projecteuclid.org:443/euclid.ijm/1255455117},
}

\bib{Hein}{article}{
      author={Heineken, H.},
       title={Eine {Bemerkung} {{\"u}ber} {Engelsche} {Element}},
        date={1960},
     journal={Archiv der Mathematik},
      volume={11},
       pages={321},
}

\bib{nekrashevych}{book}{
      author={Nekrashevych, V.},
       title={Self-similar groups},
      series={Mathematical Surveys and Monographs},
   publisher={American Mathematical Society},
        date={2005},
        ISBN={9780821838310},
         url={https://books.google.es/books?id=QG9f1Dn\_RlwC},
}

\bib{ml}{book}{
      author={Noce, M.},
       title={The first {Grigorchuk} group},
      series={Master Thesis},
   publisher={University of Salerno},
        date={2016},
}

\bib{MT}{article}{
      author={{Noce}, M.},
      author={{Tortora}, A.},
       title={{A note on Engel elements in the first Grigorchuk group}},
        date={2018-02},
     journal={ArXiv e-prints},
      eprint={1802.09032[math.GR]},
}

\bib{Gunnar}{incollection}{
      author={Traustason, G.},
       title={{Engel} groups},
        date={2011},
   booktitle={Groups {S}t {A}ndrews 2009 in {B}ath, {V}olume 2, {C}ambridge
  {U}niversity {P}ress},
      series={London Math. Soc. Lecture Note Ser.},
      volume={388},
   publisher={Cambridge University Press, Cambridge},
       pages={520\ndash 550},
}

\bib{Jone}{article}{
      author={{Uria-Albizuri}, J.},
       title={{On the concept of fractality for groups of automorphisms of a
  regular rooted tree}},
        date={2016},
     journal={Reports@SCM},
      volume={2},
       pages={33\ndash 44},
}

\end{biblist}
\end{bibdiv}

\end{document}